\documentclass[11pt,leqno,a4paper]{amsart}
\usepackage{amscd,amssymb}
\usepackage[all]{xy}

\title[A negative solution to Rosick\'y's problem]
{On the abelianization of derived categories and a negative solution to Rosick\'y's problem}
\author{Silvana Bazzoni}
\address{
Dipartimento di Matematica Pura e Applicata \\
Universit\'a di Padova \\
Via Trieste 63, 35121 Padova, Italy
}
\email{bazzoni@math.unipd.it}

\author{Jan \v{S}\v{t}ov\'\i\v{c}ek}
\address{
Charles University in Prague, Faculty of Mathematics and Physics \\
Department of Algebra \\
Sokolovska 83, 186 75 Praha 8, Czech Republic
}
\email{stovicek@karlin.mff.cuni.cz}


\thanks{%
First named author supported by MIUR, PRIN 2007, project ``Rings, algebras, modules and categories'' and by Universit\`{a} di Padova (Progetto di Ateneo
CPDA071244/07 ``Algebras and cluster categories'').
The second named author supported by GA\v{C}R~P201/10/P084, research project MSM~0021620839 and the DFG~Schwerpunkt SPP 1388.
}

\date{\today}

\newtheorem{Thm}{Theorem}[section]
\newtheorem{Lem}[Thm]{Lemma}
\newtheorem{Cor}[Thm]{Corollary}
\newtheorem{Prop}[Thm]{Proposition}

\newtheorem*{Prob}{Problem}
\theoremstyle{definition}
\newtheorem{Def}[Thm]{Definition}
\newtheorem{Constr}[Thm]{Construction}
\theoremstyle{remark}
\newtheorem{Rem}[Thm]{Remark}
\newtheorem{Ex}[Thm]{Example}
\DeclareMathOperator{\Hom}{Hom}

\DeclareMathOperator{\Add}{Add}
\DeclareMathOperator{\add}{add}
\DeclareMathOperator{\Ext}{Ext}
\DeclareMathOperator{\End}{End}

\DeclareMathOperator{\Ker}{Ker}
\DeclareMathOperator{\Coker}{Coker}

\DeclareMathOperator{\fp}{fp}
\DeclareMathOperator{\fl}{finlen}

\DeclareMathOperator{\pd}{proj.dim}
\DeclareMathOperator{\lppd}{\lambda-pure\ proj.dim}
\DeclareMathOperator{\lpgldim}{\lambda-pure\ gl.dim}

\newcommand{\Mod}{\mathrm{Mod}\textrm{-}}
\newcommand{\rmod}{\mathrm{mod}\textrm{-}}
\newcommand{\lMod}{\textrm{-}\mathrm{Mod}}

\newcommand{\Der}[1]{\mathbf{D}({#1})}
\newcommand{\KronAlg}{k(\cdot\!\rightrightarrows\!\cdot)}
\newcommand{\powS}{k[\![x]\!]}
\newcommand{\powSxy}{k[\![x,y]\!]}
\newcommand{\Ab}{\mathbf{Ab}}

\newcommand*{\la}{\longrightarrow}

\newcommand*{\cclass}[1]{\mathcal{#1}}
\newcommand*{\p}{p}

\newcommand*{\A}{\cclass{A}}
\newcommand*{\B}{\cclass{B}}
\newcommand*{\C}{\cclass{C}}

\newcommand*{\E}{\cclass{E}}

\newcommand*{\G}{\cclass{G}}

\newcommand*{\N}{\cclass{N}}

\newcommand*{\ES}{\cclass{S}}
\newcommand*{\T}{\cclass{T}}

\newcommand*{\NN}{\mathbb{N}}
\newcommand*{\Z}{\mathbb{Z}}

\renewcommand{\iff}{if and only if }
\newcommand{\st}{such that }
\newcommand{\wrt}{with respect to }

\begin{document}

\begin{abstract}
We prove for a large family of rings $R$ that their $\lambda$-pure global dimension is greater than one for each infinite regular cardinal~$\lambda$. This answers in negative a problem posed by Rosick\'y. The derived categories of such rings then do not satisfy the Adams $\lambda$-representability for morphisms for any~$\lambda$. Equivalently, they are examples of well generated triangulated categories whose $\lambda$-abelianization in the sense of Neeman is not a full functor for any~$\lambda$. In particular we show that given a compactly generated triangulated category, one may not be able to find a Rosick\'y functor among the $\lambda$-abelianization functors.
\end{abstract}

\maketitle

% ==============================================================================
\section*{Introduction}
\label{sec:intro}

Our main goal is to give simple examples of triangulated categories where an attempt to apply Neeman's representability theorems from~\cite{Nee2}, which are based on existence of Ro\-sic\-k\'{y} functors, hits serious obstacles. We achieve this by answering in negative a problem posed by Rosick\'y~\cite{RosT}:

\begin{Prob}
Given a Grothendieck category $\G$, is there an infinite regular cardinal $\lambda$ \st the $\lambda$-pure global dimension of $\G$ is at most one?
\end{Prob}

We recall that~\cite{Nee2} is a part of a program founded in~\cite{Nee}, which seeks to build up theory of triangulated categories without using models. This has proved to be a very useful approach which has found several applications in algebraic topology, representation theory, K-theory or algebraic geometry. An important ingredient of the theory are  homological functors $H\colon \T \to \A$ from a triangulated category $\T$ to an abelian category $\A$, enjoying suitable universal properties. The category $\A$ can be viewed as an ``abelian approximation'' of $\T$, and we call such a functor $H$ an abelianization of $\T$.

\smallskip

Let us be a little more specific. Partial negative results regarding Rosick\'y's problem above have already been available. Namely, Trlifaj~\cite{Tr} has proved that there is no such $\lambda$ for locally Grothendieck categories and Braun and G\"{o}bel~\cite{BrGoe} have shown that the only possible cardinality for the category of abelian groups is $\lambda = \aleph_0$. Here we give a complete negative solution by showing that the $\lambda$-pure global dimension of $\G = \Mod R$ is strictly greater than one for all $\lambda$ for a wide range of rings, including
\[
R = \mathbb{C}[x,y] \qquad \mathrm{and}
\qquad R = \begin{pmatrix}
\mathbb{C} & \mathbb{C}^2             \\
0          & \mathbb{C}\phantom{^2}   \\
\end{pmatrix}.
\]

Knowing this has significance for the theory of triangulated categories. Namely, given a ring $R$ as above, the derived category $\T = \Der{R}$ is a compactly generated triangulated category. It has proved very useful to study such categories using the category $\A_{\aleph_0}(\T)$ of all contravariant additive functors $\T^c \to \Ab$. Here, $\T^c \subseteq \T$ is the full subcategory of all compact objects and the abelianization functor is $H_{\aleph_0}\colon \T \to \A_{\aleph_0}(\T)$ sending $X$ to $\Hom_\T(-,X) |_{\T^c}$. $H_{\aleph_0}$ is usually not faithful, but as Neeman proved in~\cite[\S5]{Nee3}, inspired by classical results in algebraic topology~\cite{Bro,Adams}, it is full in some important cases. It would be favorable if $H_{\aleph_0}$ were full for all such categories, among others because of the representability theorems from~\cite{Nee2}, but it is not for $\T = \Der{\mathbb{C}[x,y]}$ as Neeman also showed in~\cite{Nee3}.

Later in~\cite{Nee}, together with founding the theory of well generated triangulated categories, Neeman discovered a good notion for the $\lambda$-abelianization $H_\lambda\colon \T \to \A_\lambda(\T)$ for any infinite regular cardinal $\lambda$. Roughly speaking, the larger $\lambda$ is, the more information $\A_\lambda(\T)$ retains about $\T$. It seemed therefore natural when a result appeared in the literature~\cite{Ros,RosE} claiming that there would always be a large enough $\lambda$  \st $H_\lambda$ would be full. In fact, such a result would be very welcome, since $H_\lambda$ would then be what Neeman calls a Rosick\'y functor, and the existence of a Rosick\'y functor for a given triangulated category has striking consequences. We again refer to Neeman's~\cite{Nee2} for details. Unfortunately, there was a gap in the argument about fullness of $H_\lambda$.

What we show here implies that $H_\lambda$ may not be full for any regular $\lambda$ and that the derived categories of the rings above are examples. Surprisingly, the situation is even more intricate in that for $\T = \Der\Ab$, the only $\lambda$-abelianization functor which is actually full is $H_{\aleph_0}$.

To summarize, we rule out the natural candidates for Rosick\'y functors for many innocent looking triangulated categories. Thus, if Rosick\'y functors do exist in broader generality, one presumably needs to develop new techniques to construct them.

\smallskip

Let us now spend a few words on how the paper is organized. In Section~\ref{sec:balanced} we extend some classical results and constructions from the theory of abelian $\p$-groups, where $\p$ is a prime number, to all discrete valuation domains. In Section~\ref{sec:adic_topology}, we use the $\p^\lambda$-adic topology to get a lower bound on $\lambda$-pure projective dimension. In Section~\ref{sec:pproj_dim} we show how to transfer this lower bound to other module categories and give first examples of module categories which do not have the property predicted by Rosick\'y. The paper is concluded by two independent sections. Section~\ref{sec:examples} gives a recipe to construct many more examples which have $\lambda$-pure global dimension at least two for all regular cardinals $\lambda$, while in Section~\ref{sec:conseq} we discuss the consequences of our results for the theory of triangulated categories.

\subsection*{Acknowledgments}
The first named author is indebted to Luigi Salce for pointing out the results on the completion of the $\p^{\lambda}$-adic topology on abelian $\p$-groups.
The second named author would like to thank Adam-Christiaan van Roosmalen and Jan Trlifaj for helpful discussions and especially for some very relevant references.

% ==============================================================================
\section{Balanced exact sequences and Walker's modules $P_\beta$}
\label{sec:balanced}

This section gives an account on and extends results by Walker~\cite{W} and Salce~\cite{S}; we also refer to~\cite{Fu} for more background. We assume throughout the section that $R$ is a discrete valuation domain and $\p \in R$ is a prime in $R$, which is then unique up to multiplication by a unit. 
All we do in this section can be in fact extended to (in general non-commutative!) hereditary noetherian uniserial rings; see Section~\ref{sec:examples}. The most important examples for us here are:

\begin{itemize}
 \item $\powS$, the ring of formal power series over a field $k$ (with $\p = x$),
 \item $\hat\Z_\p$, the $\p$-adic completion of the ring of integers for a prime $\p \in \Z$.
\end{itemize}

We will study infinitely generated torsion $R$-modules. For $R = \hat\Z_\p$, we precisely reconstruct the setting of~\cite{W,S}, where abelian $\p$-groups are in the focus.

Using the notation from~\cite[\S XI.65]{Fu}, for every $R$-module $G$ and every ordinal $\sigma$ we inductively define:
\[
\p^{\sigma}G =
\begin{cases}
G & {\rm if\ } \sigma = 0\\
\p (\p^{\sigma-1}G), & {\rm if\ } \sigma\ {\rm is\ non\ limit.}\\
\bigcap\limits_{\rho<\sigma}\p^{\rho}G, & {\rm if\ } \sigma\ {\rm is\ limit.}
\end{cases}
\]

Note that $\p G$ is under our assumption just the Jacobson radical of $G$. Using the customary notation, we further put
\[ G[r] = \{x \in G \mid rx = 0\}, \]
for any $r \in R$. Note that $G[\p]$ is none other than the socle of $G$. Given $n \in \NN$, we will also write $\p^\sigma G[\p^n]$, meaning $(\p^\sigma G)[\p^n]$.

The \emph{length} $l(G)$ of a module $G$ is the minimum ordinal $\lambda$ such that $\p^{\lambda+1}G=\p^{\lambda}G$. If $l(G)=\lambda$, then $\p^{\lambda}G$ is divisible as an $R$-module and coincides with the unique maximal divisible submodule $d(G)$ of $G$. Thus, $G$ is \emph{reduced} (i.e. has no divisible submodules), if and only if $\p^{\lambda}G=0$.

If $G$ is an $R$-module and $x\in G$, the \emph{height} $h_G(x)$ of $x$ in $G$ is defined by:
\[
h_G(x) =
\begin{cases}
\sigma & {\rm if\ } x\in \p^{\sigma}G\setminus \p^{\sigma+1}G \\
\infty & {\rm if\ } x\in \bigcap\limits_{\sigma}\p^{\sigma}G
\end{cases}
\]

Now we recall important concepts introduced by Kulikov \cite{Kul}.

\begin{Def} \label{def:proper_and_nice}
Let $G$ be a torsion $R$-module. An element $x\in G$ is \emph{proper} with respect to a submodule $N$ if $h_G(x)\geq h_G(x+u)$ for every $u\in N$.
A submodule $N$ of $G$ is \emph{nice} if every coset $x+N$ contains an element proper with respect to $N$.
\end{Def}

We recall that proper elements \wrt a submodule are characterized be the following easy lemma.

\begin{Lem} Let $N$ be a submodule of a torsion module $G$ over a discrete valuation domain.
Let $x\in G$ with $h_G(x)=\sigma$. Then, $x$ is proper with respect to $N$ if and only if $x\notin \p^{\sigma+1}G+N$.
\end{Lem}

\begin{proof}
Obvious, see also~\cite[\S XII.77]{Fu}.
\end{proof}

Although it may not be clear directly from the definition, nice submodules are related to exactness of the functors $\p^\sigma(-): \Mod R \to \Mod R$. Namely, by the same argument as in~\cite{Fu} one obtains:

\begin{Lem} \label{lem:nice_char}
A submodule $N$ of a torsion module $G$ over a discrete valuation domain is nice \iff $\p^{\sigma}(G/N)=(\p^{\sigma}G+N)/N$ for all ordinals~$\sigma$.
\end{Lem}

\begin{proof}
Exactly the same argument as in~\cite[Lemma 79.2]{Fu} applies.
\end{proof}

For further considerations, we will need a weaker version of nice submodules, where only ordinals smaller than some fixed $\lambda$ are taken into account. In view of the former lemma, we adopt the following terminology.

\begin{Def} \label{def:lambda_nice}
For an ordinal $\lambda$, a submodule $N$ of a torsion module $G$ is \emph{$\lambda$-nice} if for every $\sigma<\lambda$:
\[\p^{\sigma}(G/N)=(\p^{\sigma}G+N)/N;\]
thus $N$ is nice if and only if it is $\lambda$-nice for every ordinal $\lambda$.
\end{Def}

As yet, we have described the class of epimorphisms $f: G \to H$ \st $\p^\sigma f$ is also en epimorphism for all $\sigma$ or all $\sigma<\lambda$. This happens precisely when $\Ker f$ is nice, resp.\ $\lambda$-nice. Now we recall similar definitions describing when $\p^\sigma(-)$ commutes with kernels; see~\cite[\S XII.80]{Fu} for the case of $\p$-groups. Only the ``$\lambda$-bounded'' versions are given here; they are exactly what we need.

\begin{Def} \label{def:balanced}
Let $\lambda$ be an ordinal number and $G$ a torsion module over a discrete valuation domain. A submodule $N$ of $G$ is said to be \emph{$\lambda$-isotypic} if $\p^{\sigma}G\cap N= \p^{\sigma}N$, for every $\sigma<\lambda$.

A submodule $N$ of $G$ is called \emph{$\lambda$-balanced} if it is $\lambda$-isotypic and $\lambda$-nice, i.e.
\[\p^{\sigma}G\cap N= \p^{\sigma}N \quad {\rm and} \quad \p^{\sigma}(G/N)=(\p^{\sigma}G+N)/N\quad \forall \sigma<\lambda.\]

A short exact sequence $0\to A\to B\to C\to 0$ is said to be \emph{$\lambda$-balanced} if the image of $A$ is $\lambda$-balanced in $B$.
\end{Def}

Note that $0\to A\to B\to C\to 0$ is $\lambda$-balanced \iff $0 \to \p^\sigma A \to \p^\sigma B \to \p^\sigma C \to 0$ is exact for all $\sigma<\lambda$. The next lemma gives a very useful criterion for testing whether a submodule is balanced. The proof is almost identical with the one for~\cite[Proposition 80.2]{Fu}, but we present the argument here for convenience of the reader.

\begin{Lem} \label{lem:balanced}
Let $\lambda$ be a limit ordinal and $G$ be a torsion module over a discrete valuation domain. A submodule $N$ of $G$ is $\lambda$-balanced if and only if
\[ \dfrac{\p^{\sigma}G[\p]+N}N=\p^{\sigma}\left(\dfrac GN\right)[\p] \]
for every $\sigma<\lambda$.
\end{Lem}

\begin{proof} Assume that $N$ is $\lambda$-balanced in $G$ and let $\sigma<\lambda$. Since it is always true that $(\p^{\sigma}G[\p]+N)/N\subseteq \p^{\sigma}(G/N)[\p]$,  it is enough to prove the other inclusion.
Let $x+N\in\p^{\sigma}(G/N)[\p]$. Since by assumption $\p^{\sigma}(G/N)=(\p^{\sigma}G+N)/N$, we may assume that  $x\in  \p^{\sigma}G$. Then $\p x\in N\cap \p^{\sigma+1}G= \p^{\sigma+1}N$, hence $\p x=\p y$ with $y\in\p^{\sigma}N$. Here we have used that $\sigma+1<\lambda$ because $\lambda$ is a limit ordinal. It follows that $x-y\in \p^{\sigma}G[\p]$ and $x+N=x-y+N\in  (\p^{\sigma}G[\p]+N)/N$.

Conversely, let $(\p^{\sigma}G[\p]+N)/N=\p^{\sigma}(G/N)[\p]$ for every $\sigma<\lambda$. We prove by induction on $n$ that $\p^{\sigma}(G/N)[\p^n]\subseteq (\p^{\sigma}G+N)/N$. For $n=1$ this is true by assumption. Let $n>1$ and let $x+N\in \p^{\sigma}(G/N)[\p^n]$; then $\p x+N\in  \p^{\sigma+1}(G/N)[\p^{n-1}]\subseteq (\p^{\sigma+1}G+N)/N$, by the inductive hypothesis and the fact that $\sigma+1<\lambda$. Thus, $\p x+N= y+N$ with $y\in  \p^{\sigma+1}G$, hence $y=\p z$ with $z\in  \p^{\sigma}G$. We have $x-z+N\in \p^{\sigma}(G/N)[\p]$, hence $x-z+N=x'+N$ with $x'\in \p^{\sigma}G[\p]$. Then $x+N=x'+z+N$ with $x'+z\in \p^{\sigma}G$, thus the conclusion. When taking the union over all $n \in \NN$, we get $\p^{\sigma}(G/N) \subseteq (\p^{\sigma}G+N)/N$, or in other words, $N$ is $\lambda$-nice in $G$.

It remains to show that $N$ is $\lambda$-isotypic in $G$. It is enough to prove that, if $\p^{\sigma}G\cap N=\p^{\sigma} N$, then $\p^{\sigma+1}G\cap N=\p^{\sigma+1} N$, for $\sigma<\lambda$. Let $ g\in \p^{\sigma+1}G\cap N$, then $g=\p x$ with $x\in \p^{\sigma}G$. Thus, $x+N \in \big((\p^{\sigma}G+N)/N\big)[\p]\subseteq \p^{\sigma}(G/N)[\p]$, so by assumption, $x+N\in(\p^{\sigma}G[\p]+N)/N$. Hence $x+N=x'+N$ with $x'\in \p^{\sigma}G[\p]$ and  $x-x'\in \p^{\sigma}G\cap N=\p^{\sigma}N$. So $g=\p (x-x')\in \p^{\sigma+1}N$.
\end{proof}

The crucial fact about $\lambda$-balanced sequences of torsion modules is that they can be detected by a certain set of covariant $\Hom$-functors. We will make this precise in Proposition~\ref{prop:balanced_sequence}, but first we must give a construction of modules representing these functors. We will do so by extending the construction of the so-called groups $P_\beta$ defined and studied by E.\ Walker~\cite{W}. An alternative (and different) family of modules detecting $\lambda$-balanced sequences are the so-called generalized Pr\"ufer modules described in~\cite[\S XII.81]{Fu}.

\begin{Constr}[Modules $P_\beta$] \label{constr:p_beta}
Assume we are given a discrete valuation domain $R$, a prime element $\p \in R$, and an ordinal $\beta$. Then $P_{\beta}$ is the module with generators labeled by the finite sequences \[\beta\beta_1\beta_2\dots\beta_n\] of ordinals $\beta_i$ where $\beta>\beta_1>\beta_2>\dots>\beta_n$, which are subject to the relations:
\[ \p \cdot \beta\beta_1\beta_2\dots\beta_n\beta_{n+1}=\beta\beta_1\beta_2\dots\beta_n \quad {\rm and} \quad \p \cdot \beta=0.\]
\end{Constr}

Let us now establish basic properties of the modules $P_\beta$. In order to do so, define for every $\alpha\leq \beta$ a submodule $X_{\alpha}$ of $P_{\beta}$ by
\[ X_{\alpha}=\left \langle \beta\beta_1\beta_2\dots\beta_n\alpha\mid \beta>\beta_1>\dots>\beta_n>\alpha\right\rangle. \]
and for every $\gamma<\alpha\leq \beta$ define
\[
S_{\gamma\alpha}=\{\beta\beta_1\beta_2\dots\beta_n\gamma\mid \beta_n\geq \alpha \}
\quad {\rm and} \quad \kappa_{\gamma\alpha}= \left|S_{\gamma\alpha}\right|.
\]
The following properties are then not difficult to prove.

\begin{Lem} \label{lem:prop_p_beta} \cite{W}
For every $\alpha\leq \beta$ we have
\begin{enumerate}
\item $\dfrac{P_{\beta}}{X_{\alpha}}\cong \bigoplus\limits_{\gamma<\alpha} P_{\gamma}^{(\kappa_{\gamma\alpha})}$. In particular, $\dfrac{P_{\beta}}{\langle \beta \rangle}=\bigoplus\limits_{\gamma<\beta}P_{\gamma}$.
\item $\p^{\alpha}P_{\beta}=X_{\alpha}$. In particular, $\p^\beta P_{\beta} = \langle \beta \rangle \cong R/(\p)$ and $\p^{\beta+1}P_{\beta}=0$, so $P_\beta$ is reduced and its length equals $\beta+1$.
\item If $\beta=\alpha+\gamma$ (the ordinal sum), then $\p^{\alpha}P_{\beta}=P_{\gamma}.$
\end{enumerate}
\end{Lem}

Let us state a crucial observation, whose counterpart for generalized Pr\"ufer groups has been given by Nunke, \cite[Lemma 81.7]{Fu}.

\begin{Lem}\label{lem:morphism} \cite{W}
Let $G$ be a torsion module over a discrete valuation domain and let $g\in \p^{\beta}G[\p]$. Then there is a morphism $f\colon P_{\beta}\to G$ such that $f(\beta)=g$.
\end{Lem}

\begin{proof} It is enough to define $f$ on the generators $\beta\beta_1\beta_2\dots\beta_n$ of $P_{\beta}$ so that $f$ satisfies the relations. Let $f(\beta)=g$ and define $f(\beta\beta_1\beta_2\dots\beta_n)$ by induction on $n$. Suppose we have already defined $f(\beta\beta_1\beta_2\dots\beta_n)\in \p^{\beta_n}G$ and  consider $\beta\beta_1\beta_2\dots\beta_n\beta_{n+1}$; since $\beta_{n+1}+1\leq \beta_n$ there is $x\in \p^{\beta_{n+1}}G$ such that $f(\beta\beta_1\beta_2\dots\beta_n)=\p x$. Put $f(\beta\beta_1\beta_2\dots\beta_n\beta_{n+1})=x$. Then $f$ satisfies the relations including the relation $\p\beta=0$.
\end{proof}

Now we are able to state and prove the main result of the section. The argument here is taken from~\cite{W}, while a version for generalized Pr\"ufer groups also exists, see~\cite[Exercise 81.12]{Fu}.

\begin{Prop} \label{prop:balanced_sequence} \cite{W}
Let $\lambda$ be a limit ordinal. An exact sequence $0\to A\to B\overset{\pi}\to C\to 0$ of torsion modules over a discrete valuation domain is $\lambda$-balanced if and only if for every $\beta<\lambda$ every morphism $f\colon P_{\beta}\to C$ can be lifted to $B$. Equivalently, the sequence is $\lambda$-balanced \iff
\[ 0\to \Hom_R(P_\beta, A)\to \Hom_R(P_\beta, B)\to \Hom_R(P_\beta, C)\to 0 \]
is exact for each $\beta<\lambda$.
\end{Prop}

\begin{proof} Assume that the lifting property holds and let $b+A\in \p^{\sigma}(B/A)[\p]$, for $\sigma<\lambda$. Let $f\colon P_{\sigma}\to B/A$ be a morphism defined as in Lemma~\ref{lem:morphism}, so that $f(\sigma)=b+A$. By assumption $f$ lifts to $f'\colon P_{\sigma}\to B$. Then $f'({\sigma})\in \p^{\sigma}B[\p]$ and $f'(\sigma)+A=b+A$. Thus,  $(\p^{\sigma}B[\p]+A)/A=\p^{\sigma}(B/A)[\p]$ and by Lemma~\ref{lem:balanced}, the sequence is $\lambda$-balanced.

Conversely, assume that the sequence is $\lambda$-balanced and let $f\colon P_{\beta}\to B/A$ be a morphism for $\beta<\lambda$ . We define a lifting $f'\colon  P_{\beta}\to B$ by induction on $\beta$. If $\beta=0$, then $P_{\beta} \cong R/(\p)$ and $f(\beta)=b+A\in (B/A)[\p]= (B[\p]+A)/A$, by Lemma~\ref{lem:balanced}. Thus, there is $b'\in B[\p]$ \st $b+A = b'+A$ and $f'(\beta)= b'$ gives a lifting of $f$. Assume next that we can lift any morphism $\mu: P_\alpha \to B/A$ with $\alpha<\beta$ and let $f(\beta)=b+A$. Then $b+A\in \p^{\beta}(B/A)[\p]= (\p^{\beta}B[\p]+A)/A$, again by Lemma~\ref{lem:balanced}, so we may assume that $b\in \p^{\beta}B[\p]$. By Lemma~\ref{lem:morphism} there is a morphism $g\colon P_{\beta}\to B$ such that $g(\beta)=b$. Then $\beta \in \Ker (f-\pi g)$. Denoting by $\xi$ the projection $P_{\beta}\to P_{\beta}/\langle \beta\rangle$, we have $f-\pi g=\mu\xi$ for a morphism $\mu\colon P_{\beta}/\langle \beta\rangle\to B/A$. Since $P_{\beta}/\langle \beta\rangle\cong \bigoplus_{\gamma<\beta}P_{\gamma}$ by Lemma~\ref{lem:prop_p_beta}(1), $\mu$ can be lifted to $\mu'\colon P_{\beta}/\langle \beta\rangle\to B$ using the inductive hypothesis. Let now
\[ f'=g+\mu'\xi. \]
Then $\pi f' = \pi g+\pi\mu'\xi = \pi g+\mu\xi = \pi g+(f-\pi g) = f$.
\end{proof}

% Not completely necessary in this context
%
%\begin{Prop}\label{P:balanced-presentation} For every reduced $\p$-group $G$ there is a balanced sequence
%\[0\to K\to T\to G\to 0\]
%with $l(T)\leq l(G)$ and $T$ is a summand of a direct sum of groups $P_{\beta}$', for $ \beta<l(G)$.
%\end{Prop}
%
%\begin{proof} Let $T=\bigoplus\limits_{\beta<l(G)}P_{\beta}^{(\Hom(P_{\beta}, G))}$ and define $\phi\colon T\to G$ in the natural way. Then, $\phi$ is an epimorphism since for every $n\geq 1$, $\Z(\p^n)$ is a direct summand of some $P_{m}$. Note that if $\gamma\geq l(G)$ and $f\in \Hom(P_{\gamma}, G)$, then $f$ factors through $P_{\gamma}/\p^{l(G)}P_{\gamma}$ and by the properties above $P_{\gamma}/\p^{l(G)}P_{\gamma}$ is a direct sum of $P_{\beta}$' for $\beta<l(G)$.  Thus, by Proposition~\ref{prop:balanced_sequence}, $K$ is balanced in $T$.
%\end{proof}
%

\begin{Rem} \label{rem:history_p-groups}
Let us give a few historical comments and mention the context of the results in this section. They originate in attempts to classify infinite abelian $\p$-groups, i.e.\ for $R = \hat\Z_\p$, but work for any discrete valuation domain $R$. Countable $\p$-groups are classified by Ulm's theorem~\cite[\S XII.77]{Fu}, while for general $\p$-groups the task seems to be rather hopeless. 

There is, however, a large class of $\p$-groups which can be classified. Namely, a $\p$-group is called \emph{balanced projective} if it is projective with respect to every balanced (= $\lambda$-balanced for all $\lambda$) short exact sequence. Then the following hold:
\begin{enumerate}
\item[(1)] By Proposition~\ref{prop:balanced_sequence}, every group in $\Add\{P_{\beta}\mid \beta\ {\rm ordinal}\}$ is balanced projective (for a class $\C$ of modules $\Add \C$ denotes the class consisting of all the direct summands of direct sums of modules in $\C$).
\item[(2)] It can be shown that divisible $\p$-groups are balanced projective.
\end{enumerate}
Balanced projective groups of type (1) are called \emph{totally projective} (see~\cite[\S XII.82]{Fu}). The reason for this terminology comes from the original definition given by Nunke ~\cite{Nun, Nun2} in terms of homological properties. Totally projective groups admit an equivalent definition due to Paul Hill~\cite{Hill2, Hill1}; they are exactly the $p$-groups possessing a system $\N$ of nice subgroups satisfying:
\begin{itemize}
\item $0\in \N$ and $\sum N_i\in \N$ for every family $\{N_i\}$ of groups in $\N$;
\item if $N\in \N$ and $N\leq H\leq G$ with $H/N$ countable, then there is $N'\in \N$ such that $H\leq N'$ and $N'/N$ is countable.
\end{itemize}
As a consequence noted by Hill, totally projective groups are classified by their Ulm-Kaplansky invariants. Note that every countable reduced $\p$-group is totally projective. The following results also hold, see~\cite{W}:
\begin{enumerate}
\item[(3)] A reduced $\p$-group $G$ is totally projective \iff it belongs to $\Add\{P_{\beta} \mid \beta<l(G)\}$.
\item[(4)] Combining (1), (2) and (3) we have that a $\p$-group $G$ is balanced projective \iff $G=D\oplus T$ with $D$ divisible and $T$ totally projective.
\end{enumerate}
We recall another characterization of totally projective groups due to Crawley and Hales~\cite{CH1, CH2}.
\begin{enumerate}
\item[(5)] A reduced $\p$-group $G$ is totally projective \iff it is  \emph{simply presented}, that is $G$ can be generated by a set of elements $\{x_i\}_{i\in I}$ subject only to relations of the form:
\[\p^mx_i=0\qquad {\rm or}\quad \p^nx_i=x_j, \quad i\neq j, n,m\in \NN.\]
\end{enumerate}
\end{Rem}

% ==============================================================================
\section{An application of the $\p^{\lambda}$-adic topology}
\label{sec:adic_topology}

Also in this section we assume that $R$ is a discrete valuation domain and $\p \in R$ is a prime. Moreover, we fix a limit ordinal $\lambda$. Our aim is to extend to our setting the concept of the $\p^{\lambda}$-adic topology on an abelian group, which has been investigated by R.\ Mines~\cite{Mi}. As a consequence, we will show that for any uncountable regular cardinal $\lambda$, the $\lambda$-pure global dimension of $\Mod R$ is at least two.

The $\p^{\lambda}$-adic topology on an $R$-module $G$ is a linear topology with basis of neighborhoods of zero taken as the family of submodules $\{\p^{\sigma}G\ \mid \sigma<\lambda\}$. If $G$ is a torsion $R$-module, the canonical morphism
\[ \delta_{\lambda}\colon G\la \prod_{\sigma<\lambda}G/\p^{\sigma}G, \]
has kernel $\p^{\lambda}G$ and its image is contained in the torsion submodule $T_{\lambda}(G)$ of the completion $L_{\lambda}(G)$, where $L_{\lambda}(G)=\varprojlim_{\sigma<\lambda}G/\p^{\sigma}G$, viewed canonically as a submodule of $\prod_{\sigma<\lambda}G/\p^{\sigma}G$.

Let us now assume that $G$ is also reduced, that is, it contains no non-zero divisible submodule. 
Note that then we have:
\begin{itemize}
\item $G$ is discrete in the $\p^{\lambda}$-adic topology \iff $l(G)<\lambda$.
\item $G$ is a Hausdorff topological space \iff $l(G)\le\lambda$.
\end{itemize}

Since our main aim is to study torsion $R$-modules, which are analogues of $\p$-groups in our generalized setting, we would like to know the relation of $T_{\lambda}(G)$ and $L_{\lambda}(G)$. The following result says that often the situation is as favorable as it can be, since the completion of a torsion module is torsion again.

\begin{Prop} \label{prop:not-cofinal} \cite{Mi}
If $\lambda$ is an ordinal of uncountable cofinality, then for every torsion module $G$ we have $T_{\lambda}(G)=L_{\lambda}(G).$
\end{Prop}
\begin{proof} Let $ \xi=(x_{\sigma}+\p^{\sigma}G)_{\sigma<\lambda} \in L_{\lambda}(G)$. If $\xi$ is not a torsion element, then for every $k\in \NN$ there is an ordinal $\sigma_k$ such that $\p^k(x_{\sigma}+\p^{\sigma}G)\neq 0$, for every $\sigma \ge \sigma_k$.
Let $\alpha = \sup \{\sigma_k\mid k\in \NN\}$; by assumption $\alpha<\lambda$. For every $\sigma \ge \alpha$ and $k \in \NN$ we have $\p^k(x_{\sigma}+\p^{\sigma}G)\neq 0$, which is a contradiction, since the elements $x_{\sigma}$ are torsion.
\end{proof}

The following proposition is crucial, it gives us, provided $\lambda$ has uncountable cofinality, a method to construct a plethora of non-trivial examples of $\p^\lambda$-adic complete modules. The idea for its proof uses classical arguments. We follow here the presentation given in \cite{S}.
 
\begin{Prop} \label{prop:completeness} \cite{S}
Let $\lambda$ be an ordinal of uncountable cofinality and for every $\alpha<\lambda$, let $G_{\alpha}$ be a reduced torsion $R$-module of length $\leq \alpha$. Then $\bigoplus_{\alpha<\lambda}G_{\alpha}$ is complete in the $\p^{\lambda}$-adic topology.
\end{Prop}

\begin{proof}
Let $G=\bigoplus\limits_{\alpha<\lambda}G_{\alpha}$ and let $L=\prod\limits_{\alpha<\lambda}G_{\alpha}$. Let $\tau$ be the $\p^{\lambda}$-adic topology on $L$. Then $\tau$ is finer than the product topology of the discrete topologies on the $G_{\alpha}$. Indeed, if $U$ is a neighborhood of zero in the product topology, then $U$ contains the direct product of $G_{\alpha}$ for $\alpha\notin F$ with $F$ a finite subset of $\lambda$. Taking $\beta$ the maximum of $F$ we have that $\p^{\beta}G_{\alpha}=0$ for every $\alpha\leq \beta$, by the assumption $l(G_{\alpha})\leq \alpha$; hence $\p^{\beta}L\subseteq U$. Moreover, the submodules $\p^{\sigma}L$, for $\sigma<\lambda$, are closed in the product topology, since they are products of closed submodules $\p^{\sigma} G_\alpha \subseteq G_\alpha$.
By \cite[Corollaire 2, III.26, par. 3]{Bou}, $L$ is complete in the topology $\tau$.
 
Now for every $\sigma<\lambda$ we have $\p^{\sigma}L\cap G=\p^{\sigma}G$, so $\tau$ induces the  $\p^{\lambda}$-adic topology on $G$. To conclude that $G$ is complete in the $\p^{\lambda}$-adic topology it is enough to prove that $G$ is  $\tau$-closed in $L$. Assume by way of contradiction that $G$ is not closed in $L$. Then there exists an element $\xi=(x_{\alpha})_{\alpha<\lambda}\in L$ such that $\xi$ is in the closure of $G$ but not in $G$. This means that,  for every $\sigma<\lambda$ there is an element $y^{\sigma}\in G$ such that $\xi-y^{\sigma}\in \p^{\sigma}L$. Since $\xi\notin G$, there is an increasing sequence $\beta_1<\beta_2<\dots $ of ordinals such that, for every $n\in \NN$, all the $\beta_n$-components $x_{\beta_n}$ of $\xi$ are nonzero. Choose $\beta<\lambda$, $\beta> \beta_n$ for every $n$; then $\xi-y^{\beta}\in \p^{\beta}L $. Checking the $\beta_n$-components we have $x_{\beta_n}-y^{\beta}_{\beta_n}\in \p^{\beta}G_{\beta_n}=0$, since $l(G_{\beta_n})\leq \beta_n<\beta$. Thus, $y^{\beta}_{\beta_n}\neq 0$  for all $n$, which is a contradiction since $y^{\beta}\in G$ can have only finitely many nonzero components.
\end{proof}
To state further results, we need to recall a few facts about $\lambda$-pure submodules and exact sequences. For a regular cardinal $\lambda$, a submodule $A$ of an $R$-module $B$ is said to be \emph{$\lambda$-pure} in $B$ if for every $<{}\lambda$-presented $R$-module $Y$, every morphism $f\colon Y\to B/A$ can be lifted to $B$. A short exact sequence
\[ 0\la A\la B\la C\la 0 \]
is said to be \emph{$\lambda$-pure} if the image of $A$ is $\lambda$-pure in $B$. Thus the usual notion of purity is the $\aleph_0$-purity in our terminology.

A well-known and important fact is that $\Mod R$ together with the class of all $\lambda$-pure exact sequences forms an exact category in the sense of Quillen; we refer to~\cite[Appendix A]{Kel} and~\cite{Bu} for more details on the concept. A practical consequence is that we can define a relative version of Yoneda Ext-functors and do homological algebra. In fact, another well known fact is that we have enough projectives: A module is \emph{$\lambda$-pure projective} \iff it is a summand in a direct sum of a family of $<{}\lambda$-presented modules. As an immediate consequence of the previous proposition we then have:

\begin{Prop}\label{prop:direct_sums}
Let $\lambda$ be an uncountable regular cardinal. Then any reduced torsion $\lambda$-pure projective $R$-module is complete in the $\p^{\lambda}$-adic topology.
\end{Prop}
\begin{proof}
By definition, a $\lambda$-pure projective module $G$ is a direct summand in a sum of $<{}\lambda$-presented, or here, since any discrete valuation domain is noetherian, equivalently $<\lambda$-generated modules. Using the Kaplansky-Walker theorem~\cite[26.1]{AF}, we infer that we have $G \cong \bigoplus_{i\in I}G_i$ with each $G_i$ $<{}\lambda$-generated.

Then every $G_i$ has length less than $\lambda$. Indeed, $G_i$ is a directed union of a family $(G_i^j \mid j \in J)$ of its finitely generated submodules \st $\left|J\right| < \lambda$. Since $G_i$ is a torsion module, so is each $G_i^j$. Thus, each $G_i^j$ is of finite length. Now one easily sees that given any ordinal $\sigma$, we have $\p^\sigma G_i \supsetneqq \p^{\sigma+1} G_i$ \iff $(G_i^j \cap \p^\sigma G_i) \supsetneqq (G_i^j \cap \p^{\sigma+1} G_i)$ for some $j \in J$. However, each $G_i^j$ can occur only for finitely many ordinals $\sigma$, so $l(G_i) < \lambda$ as claimed.

Finally, for every $\alpha<\lambda$ let $I_{\alpha}=\{i\in I\mid l(G_i)=\alpha\}$ and define $G_{\alpha}=\bigoplus_{i\in I_{\alpha}}G_i$.
Then $G \cong \bigoplus_{\alpha<\lambda}G_{\alpha}$ and we conclude by applying Proposition~\ref{prop:completeness}.
\end{proof}

Using the $\lambda$-pure exact structure on $\Mod R$, we can define \emph{$\lambda$-pure projective dimension} of a module and \emph{$\lambda$-pure global dimension} of $\Mod R$ in the obvious way. With the terminology, our goal in this section is to prove that if $\lambda$ is uncountable, the $\lambda$-pure projective dimension of $P_\lambda$ is at least $2$. We start with giving a suitable presentation of $P_\lambda$.

\begin{Prop} \label{prop:presentation}
Let $\lambda$ be an uncountable regular cardinal and $P_\lambda$ the module as in Construction~\ref{constr:p_beta}. Then there is a $\lambda$-balanced exact sequence
\[ 0\la K\la T\la P_{\lambda}\la 0 \]
such that $T=\bigoplus\limits_{\beta<\lambda} P_{\beta}^{(\Hom(P_{\beta}, P_{\lambda}))}$ and $K$ is not complete in the $\p^{\lambda}$-adic topology.
\end{Prop}
\begin{proof} Put $T=\bigoplus_{\beta<\lambda} P_{\beta}^{(\Hom(P_{\beta}, P_{\lambda}))}$ and consider the obvious morphism $\phi\colon T\to P_{\lambda}$. Note that  $\phi$ is surjective since for each $\beta<\lambda$ the map $\nu_{\beta}$ defined by $\nu_{\beta}( \beta\beta_1\beta_2\dots\beta_n)= \lambda\beta_1\beta_2\dots\beta_n $
induces an embedding $P_\beta \to P_\lambda$. Hence every generator of $P_\lambda$ is in the image of $\nu_{\beta}$ for some $\beta$. Then $0\to K\to T\to P_{\lambda}\to 0$ is $\lambda$-balanced by Proposition~\ref{prop:balanced_sequence}.
%%There was a mistake before, because it is not true that  $R/(\p^n)$ is a summand of $P_n$.} Then $0\to K\to T\to P_{\lambda}\to 0$ is $\lambda$-balanced by Proposition~\ref{prop:balanced_sequence}.
 
By Proposition~\ref{prop:direct_sums}, $T$ is complete in the $\p^{\lambda}$-adic topology. Since $K$ is $\lambda$-balanced in $T$, the subspace topology induced on $K$ is precisely the $\p^{\lambda}$-adic topology on $K$. Thus $K$ is complete if and only if it is closed in $T$.
Denoting by $\overline K$ the closure of $K$ in $T$, we get
\[ \overline K = \bigcap_{\sigma<\lambda}(\p^{\sigma} T+K) \quad {\rm and} \quad \dfrac{\p^{\sigma} T+K}K = \p^{\sigma}\left(\dfrac{T}K\right), \]
for every $\sigma<\lambda$,  since $K$ is $\lambda$-nice in $T$. Therefore
\[
\dfrac{\overline K} K =
\bigcap_{\sigma<\lambda}\dfrac{\p^{\sigma} T+K}K =
\bigcap_{\sigma<\lambda}\p^{\sigma}\left(\dfrac{T}K\right) =
\p^{\lambda}\dfrac{T}K =
\p^{\lambda}P_{\lambda} \cong
R/(\p),
\]
so $K$ is neither closed in $T$ nor complete in the $\p^\lambda$-adic topology.
\end{proof} 
Now we can state the main result of the section. First we recall that for a regular cardinal $\lambda$ a  partially ordered set is called \emph{$\lambda$-directed} if each subset of cardinality smaller than $\lambda$ has an upper bound. The colimit of a $\lambda$-directed system is called a $\lambda$-directed colimit. 

\begin{Thm}\label{thm:K-pure}
Let $R$ be a discrete valuation domain and $\lambda$ be an uncountable regular cardinal. Then there is a $\lambda$-pure exact sequence
\[ 0\la K\la T\la P_{\lambda}\la 0\]
\st $T$ is $\lambda$-pure projective, but $K$ is not. In particular, both the $\lambda$-pure projective dimension of $P_\lambda$ and the $\lambda$-pure global dimension of $\Mod R$ are at least two.
\end{Thm}

\begin{proof}
We again use the fact that for each $\beta<\lambda$ the mapping on generators
\[ \beta\beta_1\beta_2\dots\beta_n \to \lambda\beta_1\beta_2\dots\beta_n \]
induces an embedding $P_\beta \to P_\lambda$ and $P_\lambda$ is a $\lambda$-directed union of the images of all such embeddings. Therefore, any morphism $M \to P_\lambda$ with $M$ $<\lambda$-presented factors through some embedding $P_\beta \to P_\lambda$, $\beta<\lambda$, and the short exact sequence $0\to K\to T\to P_{\lambda}\to 0$ constructed in Proposition~\ref{prop:presentation} is $\lambda$-pure. Now, $T$ is $\lambda$-pure projective by the construction, but $K$ cannot be $\lambda$-pure projective by Proposition~\ref{prop:direct_sums}, since it is not complete.
\end{proof}

\begin{Rem} \label{rem:uncnt_essential}
The assumption that $\lambda$ is uncountable is essential. The usual pure global dimension of any discrete valuation domain is one.
\end{Rem}

\begin{Rem} \label{rem:value_of_pgldim}
One may ask what is the exact value of the $\lambda$-pure global dimension of $\Mod R$ for $\lambda$ uncountable. Is it two? Is it finite? We do not know. We will explain in Section~\ref{sec:conseq} why an answer may be interesting.
\end{Rem}

% ==============================================================================
\section{The $\lambda$-pure projective dimension in accessible categories}
\label{sec:pproj_dim}

The main motivation for this paper was to find a ring whose $\lambda$-pure global dimension is greater than one for \emph{all} infinite regular cardinals $\lambda$. More details on consequences of having such a ring will be given in Section~\ref{sec:conseq}. With Theorem~\ref{thm:K-pure} we have succeeded for all cardinals but for $\aleph_0$. To remove the glitch, we need to look for other rings than discrete valuation domains. Here we will use results by Kaplansky~\cite{Kap}, Osofsky~\cite{Os} and Lenzing~\cite{Lenz}. First, however, we need to establish preliminary results in the spirit of~\cite{AR}.

\begin{Def} \label{def:accessible}
Let $\lambda$ be an infinite regular cardinal and $\A$ be a %n additive
category with $\lambda$-directed colimits. An object $X \in \A$ is called \emph{$\lambda$-presentable} provided that for any $\lambda$-directed system $(Y_i \mid i \in I)$, the canonical map
\[ \varinjlim \Hom_\A(X, Y_i) \la \Hom_\A(X, \varinjlim Y_i) \]
is an isomorphism.

The category $\A$ is called \emph{$\lambda$-accessible} if it admits a set $\ES$ of $\lambda$-presentable objects \st each $Y \in \A$ is a $\lambda$-directed colimit of objects from $\ES$.

A full subcategory $\B$ of $\A$ is said to be a \emph{$\lambda$-accessible subcategory} if there is a set $\ES' \subseteq \B$ of objects which are $\lambda$-presentable in $\A$ and \st $\B$ consists precisely of $\lambda$-directed colimits of objects from $\ES'$.

For $\lambda = \aleph_0$, we will replace the prefix $\lambda$- by the word \emph{finitely}, we speak of finitely presentable objects and finitely accessible (sub)categories.
\end{Def}

One would expect, based on the terminology, that a $\lambda$-accessible subcategory is itself a $\lambda$-accessible category.
This is indeed the case, which we show using an argument inspired by~\cite[Proposition 2.1]{Lenz2}.

\begin{Lem} \label{lem:acces_sub}
Let $\lambda$ be an infinite regular cardinal, $\A$ a $\lambda$-accessible category and $\B$ a $\lambda$-accessible subcategory given by a set $\ES'$ of $\lambda$-presentable objects as in Definition~\ref{def:accessible}. Then the following is equivalent for $Y \in \A$:
\begin{enumerate}
  \item $Y \in \B$;
  \item For any $\lambda$-presentable object $X \in \A$, any morphism $X \to Y$ factors through an object from $\ES'$. 
\end{enumerate}
In particular, $\B$ is closed under taking $\lambda$-directed colimits in $\A$ and it is itself a $\lambda$-accessible category.
\end{Lem}

\begin{proof}
We denote by $\ES$ a representative set of all $\lambda$-presentable objects in $\A$. Note that by~\cite[Theorem 1.5]{AR}, we can freely use $\lambda$-filtered colimits instead of $\lambda$-directed ones. Moreover, \cite[Proposition 2.8]{AR} tells us that $Y$ is the $\lambda$-filtered colimit of its canonical diagram \wrt $\ES$, indexed by the comma category $\ES\downarrow Y$.

Suppose that $(2)$ holds. This is nothing but to say that the comma category $\ES'\downarrow Y$ is cofinal in $\ES\downarrow Y$. Hence $Y \in \B$, since it is a $\lambda$-filtered colimit of objects from~$\ES'$. Conversely assuming $(1)$, one directly infers $(2)$, again since $Y$ is a $\lambda$-filtered colimit of objects from $\ES'$. The final statement is easy because condition (2) is preserved by $\lambda$-directed colimits in $\A$.
\end{proof}

Note that $\Mod R$ is a finitely accessible category for any ring $R$. Finitely accessible subcategories of module categories were studied for instance in~\cite{Lenz2}. An important fact is that each additive $\lambda$-accessible category is a $\lambda$-accessible subcategory of a module category if we use a suitable notion of modules. Given a small additive category $\ES$, we use $\Mod\ES$, the category of all \emph{right $\ES$-modules}, that is, the category of additive contravariant functors $\ES \to \Ab$. What we get is an additive version of the Representation Theorem~\cite[2.26]{AR}.

\begin{Prop} \label{prop:representability}
For each infinite regular cardinal $\lambda$, the following are equivalent for an additive category $\A$:
\begin{enumerate}
 \item $\A$ is $\lambda$-accessible;
 \item There is a small additive category $\ES$ \st $\A$ is equivalent to the full subcategory of $\Mod\ES$ formed by all colimits of $\lambda$-directed systems of functors of the form $\Hom_\ES(-,S)$, $S \in \ES$. 
\end{enumerate}
\end{Prop}

% The crucial step is in \cite[Proposition 2.8, p. 73]{AR}
\begin{proof}
The proof is the same as for~\cite[Theorem 2.26]{AR}, when replacing the category of sets by the category of abelian groups and functors by additive functors. Note that given (1), we can take for $\ES$ a representative set for all $\lambda$-presentable objects of $\A$. The equivalence in (2) is then given by the Yoneda embedding sending an object $A \in \A$ to the contravariant functor $\Hom_{\A}(-, A) |_{\ES}$.
\end{proof}

A natural question is whether a $\lambda$-accessible category is also $\lambda'$-accessible for another cardinal $\lambda'$. For the sake of simplicity, we only give a special case of what is called ``raising the index of accessibility'' in~\cite{AR}.

\begin{Lem} \label{lem:raising_access}
Let $\A$ be a finitely accessible category and $\B$ a finitely accessible subcategory. Then $\A$ is a $\lambda$-accessible category and $\B$ is a $\lambda$-accessible subcategory for each infinite regular cardinal $\lambda$.
\end{Lem}

\begin{proof}
This is proved in~\cite[2.11--2.13]{AR}. Let us give a sketch of the argument. Every finitely accessible category has $\lambda$-directed colimits for each infinite regular cardinal $\lambda$.
Let $\C$ be the collection of all objects of $\A$ which are directed colimits of fewer than $\lambda$ finitely presentable objects of $\A$. Then $\C$ is skeletally small and every object of $\C$ is $\lambda$-presentable. It is now enough to show that every object $A$ of $\A$ is a $\lambda$-directed colimit of objects of $\C$. Write $A$ as a directed colimit of  a directed system $(S_i \mid i \in I)$, where the $S_i$'s are finitely presentable and let $I'$ be the poset of all the directed subsets $J\subseteq I$ with cardinality less than $\lambda$, ordered by inclusion. Clearly $I'$ is $\lambda$-directed. For each $J\in I'$, let $C_J\in \C$ be the directed colimit of $(S_i \mid i \in J)$. For $J\subseteq J'$ there are canonical morphisms $C_J\to C_{J'}$ and also $C_J\to A$. It is easy to see that $A$ is the $\lambda$-directed colimit of the system $(C_J \mid J \in I')$.
\end{proof}

The crucial point for us is that we can define the $\lambda$-pure exact structure not only on a module category as in Section~\ref{sec:adic_topology}, but on any additive $\lambda$-accessible category; see also~\cite[\S 2.D]{AR}. Let us call a sequence of morphisms
\[ 0 \la A \overset{i}\la B \overset{d}\la C \la 0 \]
in $\A$ \emph{$\lambda$-pure exact} if the application of $\Hom_\A(X,-)$ gives a short exact sequence of abelian groups for all $\lambda$-presentable objects $X \in \A$. Given such a sequence, the morphism $i\colon A \to B$ is said to be a \emph{$\lambda$-pure monomorphism}, while the morphism $d\colon B \to C$ is a \emph{$\lambda$-pure epimorphism}. As a consequence of Proposition~\ref{prop:representability}, $\A$ with this class of sequences becomes again an exact category. In particular in any $\lambda$-pure sequence we have $i = \Ker d$ and $d = \Coker i$. For details on exact categories we refer to~\cite[Appendix A]{Kel} or~\cite{Bu}.

\begin{Prop} \label{prop:exact_catg}
Let $\lambda$ be an infinite regular cardinal and $\A$ an additive $\lambda$-accessible category. Then $\A$ together with the class $\E_\lambda$ of all $\lambda$-pure exact sequences satisfies the axioms of an exact category. Moreover, each morphism $d\colon B \to C$ \st $\Hom_\A(X,d)$ is surjective for all $\lambda$-presentable objects $X \in \A$ has a kernel, so it is a $\lambda$-pure epimorphism in $\A$.

If $\A$ has coproducts, it also has enough projectives \wrt this exact structure. Namely, for each $Y \in \A$ there is a $\lambda$-pure exact sequence
\[ 0 \la K \la P \overset{d}\la Y \la 0 \]
\st $P \cong \bigoplus X^{(\Hom_\A(X,Y))}$, where $X$ runs over all isoclasses of $\lambda$-presentable objects of $\A$, and $d$ is the obvious morphism.
\end{Prop}

\begin{proof}
If we view $\A$ as a full subcategory of $\Mod\ES$ as in Proposition~\ref{prop:representability}(2), then any object of $\A$ is a $\lambda$-directed colimit of $\Hom$-functors and $\Hom$-functors are projective in $\Mod\ES$. Thus, a short exact sequence in $\Mod\ES$ with the third term in $\A$ is a $\lambda$-directed colimit of split exact pull-back sequences ending in $\Hom$-functors, so it is a $\lambda$-pure exact sequence in $\Mod\ES$ as in \cite[Proposition 7.16]{JL}. A generalization to higher cardinals of the arguments in the proof of ~\cite[Proposition 2.2]{Lenz2} (taking into account Lemma~\ref{lem:acces_sub} applied to the categories $\A$ and $\Mod\ES$), shows that a morphism $d$ as in the statement has kernel in $\A$ and that $\A$ is closed under extensions. Using~\cite[Lemma 10.20]{Bu} we conclude that the class $\E_\lambda$ of all $\lambda$-pure exact sequences gives rise to an exact category. This establishes the first part.
The second part is easy, but beware: Coproducts in $\A$ may not coincide with coproducts in $\Mod\ES$.
\end{proof}

Knowing that $\A$ is an exact category with enough projectives, we can again define the projective dimension of an object and the global dimension of $\A$ as an exact category. In the case considered in Proposition~\ref{prop:exact_catg}, we speak of \emph{$\lambda$-pure projective dimension} of $Y \in \A$, denoted $\lppd_\A Y$, and of \emph{$\lambda$-pure global dimension} of $\A$, denoted $\lpgldim\A$.

If now $\B \subseteq \A$ is a $\lambda$-accessible subcategory of $\A$ closed under taking coproducts and $Y \in \B$, we can compute $\lambda$-pure projective dimension of $Y$ both in $\A$ and $\B$. It is useful to learn that the ambient category does not matter:

\begin{Lem} \label{lem:context_change}
Let $\lambda$ be an infinite regular cardinal, $\A$ an additive $\lambda$-accessible category with coproducts and $\B$ a $\lambda$-accessible subcategory closed under taking coproducts. Then for any $Y \in \B$ we have
\[ \lppd_\A Y = \lppd_\B Y. \]
In particular, $\lpgldim \B \le \lpgldim \A$.
\end{Lem}

\begin{proof}
First note that the $\lambda$-presentable objects of $\B$ are precisely the $\lambda$-presentable objects in $\A$ which belong to $\B$. To see this, let $\ES'\subseteq\B$ be a set of objects as in Definition~\ref{def:accessible} and consider $X \in \B$ expressed as a colimit of a $\lambda$-directed diagram $(S_i \mid i \in I)$ over $\ES'$. If $X$ is $\lambda$-presentable in $\B$, then $1_X$ factors through some $S_i$, so $X$ is a summand of an object of $\ES'$. Hence $X$ is $\lambda$-presentable in $\A$. The other implication is clear by Lemma~\ref{lem:acces_sub}, since $\B$ is closed under $\lambda$-directed colimits in $\A$.

Next consider the morphism $d\colon \bigoplus X^{(\Hom_\B(X,Y))} \to Y$, where $X$ runs up to isomorphism over all $\lambda$-presentable objects in $\B$. We have shown in Proposition~\ref{prop:exact_catg} that $d$ is a $\lambda$-pure epimorphism in $\B$. By Lemma~\ref{lem:acces_sub}, $d$ is also a $\lambda$-pure epimorphism in $\A$. Therefore, there is a sequence
\[ 0 \la K \la \bigoplus X^{(\Hom_\B(X,Y))} \overset{d}\la Y \la 0, \]
which is $\lambda$-pure both in $\A$ and $\B$, and the middle term is $\lambda$-pure projective also both in $\A$ and $\B$. Iterating the procedure, we see that there is a $\lambda$-pure projective resolution of $Y$ in $\B$ which is a $\lambda$-pure projective resolution in $\A$ as well, so the $\lambda$-pure projective dimensions must coincide. The rest is clear.
\end{proof}

Now we are ready to apply results of Kaplansky~\cite{Kap}, Osofsky~\cite{Os} and Lenzing~\cite{Lenz} to get one of the main results of the paper.
\begin{Thm} \label{thm:examples}
Let $k$ be a field. Then the following hold:
\begin{enumerate}
\item $\lpgldim (\Mod \KronAlg) \ge 2$ for any infinite regular cardinal $\lambda$, provided that $k$ is uncountable.
\item $\lpgldim (\Mod k[x,y]) \ge 2$ for any infinite regular cardinal $\lambda$, provided that $k$ is uncountable.
\item $\lpgldim (\Mod \powSxy) \ge 2$ for any infinite regular cardinal $\lambda$.
\end{enumerate}
\end{Thm}

\begin{proof}
(1) Lenzing proved in~\cite[Proposition 3.1]{Lenz} that the $\aleph_0$-pure projective dimension of the generic Kronecker module
\[
G\colon \quad
\xymatrix{
k(x) \ar@<.5ex>[r]^{x \cdot -} \ar@<-.5ex>[r]_{1 \cdot -} & k(x)
}
\]
is equal to two for an uncountable field $k$. That is, the pure global dimension of $\Mod\KronAlg$ is at least two (in fact it is exactly two, see~\cite[Proposition 3.9]{Lenz2}).

Fix now an uncountable regular cardinal $\lambda$ and denote by $\B$ the class of all torsion $\powS$-modules. That is, we take precisely the $\powS$-modules for which the action of $x$ is locally nilpotent. Clearly, $\B$ is a finitely accessible subcategory of $\Mod\powS$. Moreover, combining Theorem~\ref{thm:K-pure} with Lemma~\ref{lem:context_change}, we get
\[ \lppd_\B P_\lambda \ge 2 \]
where $P_\lambda$ is as in Construction~\ref{constr:p_beta}. In particular, $\lpgldim \B \ge 2$. Now, it is straightforward to check that the assignment
\[
B \quad \mapsto \quad \Big(
\xymatrix{
B \ar@<.5ex>[r]^{x \cdot -} \ar@<-.5ex>[r]_{1 \cdot -} & B
}
\Big)
\]
induces an equivalence between $\B$ and a finitely accessible subcategory $\B'$ of $\Mod\KronAlg$. Namely, denoting by $\mathbf{t}$ the tube corresponding to the regular module (see~\cite{ARS} for missing terminology)
\[
\xymatrix{
k \ar@<.5ex>[r]^{0} \ar@<-.5ex>[r]_{1} & k
}
\]
one easily checks that $\B' = \varinjlim(\add\mathbf{t})$. Invoking Lemma~\ref{lem:context_change} again, we get:
\[ \lpgldim \Mod\KronAlg \ge \lpgldim \B' \ge 2. \]

(2) \& (3) Let $R = k[x,y]$ or $\powSxy$ and $Q$ stand for the fraction field of $R$. Regarding the $\aleph_0$-pure global dimension of $\Mod R$, Kaplansky and Osofsky computed in~\cite{Kap} and~\cite[Corollary 2.59]{Os}, respectively, that the projective dimension of $Q$ is equal to two. Since $Q$ is flat, the projective and pure projective dimensions of $Q$ coincide and the pure global dimension of $\Mod R$ is at least two.

For $\lambda$ regular uncountable, we would again like to construct an equivalence functor from the category $\B$ of torsion $\powS$-modules onto a finitely accessible subcategory of $\Mod R$. But this is easy: given a torsion $\powS$-module $B$, we can view it as an $R$-module with the trivial action of $y$. Thus again $\lpgldim \Mod R \ge \lpgldim \B \ge 2$.
\end{proof}

% ==============================================================================
\section{Abundance of examples}
\label{sec:examples}

In Theorem~\ref{thm:examples}, we have presented a few examples of rings whose $\lambda$-pure global dimension is at least two for all infinite regular cardinals $\lambda$. In this section we show that these examples are not isolated, there is a vast range of others. Since $\aleph_0$-pure global dimension is well covered in the literature (see for instance~\cite{Kap,Os,Lenz,BBL1,BBL2}), we will focus on the case where $\lambda$ is uncountable. Our main result here is the following:

\begin{Thm} \label{thm:grothendieck_catg}
Let $\A$ be a locally finitely presentable Grothendieck category and $\lambda$ an uncountable regular cardinal. Suppose that $\A$ either contains a tube or is strictly wild. Then $\lpgldim\A \ge 2$.
\end{Thm}

Let us explain the terminology first. By a \emph{locally finitely presentable} Grothendieck category $\A$ we mean just a finitely accessible Grothendieck category; the terminology is used to stay coherent with~\cite{GaUl} and work of several other people.
We denote by $\fp\A$ the full subcategory of $\A$ consisting of finitely presentable objects and by $\fl\A$ the full subcategory of all objects of finite length.

Following~\cite{Gab}, we say that a Grothendieck category $\A$ is \emph{locally finite} if it admits a set of generators of finite length. This is a stronger condition than being locally finitely presentable and in that case $\fp\A = \fl\A$.

Given a locally finitely presentable Grothendieck category, one tries to measure the complexity of the category $\fp\A$. To this end, we introduce two definitions. The first is inspired by~\cite[Chapter 3]{Ri} and the definition of 1-spherical objects in~\cite{VRos}.

\begin{Def} \label{def:tube}
A locally finitely presentable Grothendieck category $\A$ is said to \emph{contain a tube} if there is an object $S \in \fp\A$ \st
\begin{enumerate}
 \item $k = \End_\A(S)$ is a skew-field;
 \item $\Ext^1_\A(S,S) \cong k$ both as left and as right $k$-modules;
 \item $\Ext^2_\A(S,S) = 0$.
\end{enumerate}
\end{Def}

\begin{Ex} \label{expl:tube}
If $\A = \Ab$, we can take $S = \Z/(\p)$ for any prime number $\p$. Given any Dedekind domain $R$ and $\A = \Mod R$, $S$ can be any simple module. Analogously, if $\A = \mathfrak{Qco}(X)$ for a projective or affine curve~$X$ and $\p \in X$ is a non-singular closed point, the corresponding simple coherent sheaf $S$ fits the definition. Finally if $R$ is a tame hereditary artin algebra and $\A = \Mod R$, any quasi-simple regular module $S$ serves as an example.
\end{Ex}

The other notion is that of wildness, which is often taken as a synonym to intractability of the problem of classifying the objects in $\fp\A$. Here we take a variant of the concept from~\cite{KL}, which works both for module categories of finite dimensional algebras and of commutative noetherian rings.

\begin{Def} \label{def:wild}
A locally finitely presentable Grothendieck category $\A$ is called \emph{strictly wild} if there is a field $k$ \st given any finite dimensional $k$-algebra $R$ (non-commutative in general), there is a fully faithful functor $\Phi\colon \rmod R \to \fp\A$.
% which preserves indecomposability (i.e. if $X \in \rmod R$ is indecomposable, so is $\Phi(X)$).
\end{Def}

\begin{Ex} \label{expl:wild}
If $Q$ is a finite quiver without oriented cycles which is wild (cf.~\cite[\S 1.3]{Kern}), then the category $\A = \Mod kQ$ is strictly wild by~\cite[Theorem 1.6]{Kern}. Examples of strictly wild module categories over a commutative noetherian ring are provided in~\cite[\S3]{KL}.
\end{Ex}

To give a proof of Theorem~\ref{thm:grothendieck_catg}, let us introduce a slight non-commutative generalization of discrete valuation domains, as promised at the beginning of Section~\ref{sec:balanced}. An easy non-commutative example of such a ring is $\mathbb{H}[\![x]\!]$, where $\mathbb{H}$ is the skew-field of quaternions.

\begin{Def} \label{def:asano}
A ring $R$, in general non-commutative, is called \emph{noetherian uniserial} (cf.~\cite[Appendix B]{CF}) if $R$ is left and right noetherian and both left and right ideals are linearly ordered.

A noetherian uniserial ring $R$ is called an~\emph{Asano ring} (cf.~\cite{AmRi2} and also~\cite[\S IV.4]{Gab}) if it is complete. That is, if $J$ is the Jacobson of $R$, we require
\[ R = \varprojlim R/J^n. \]
\end{Def}

Let us collect a few basic properties of these rings. Recall that a module is \emph{semiartinian} if it is a direct limit of finite length modules. For modules over a discrete valuation domain, semiartinian is the same as torsion.

\begin{Lem} \label{lem:asano}
The following hold for any noetherian uniserial ring $R$ and its Jacobson radical $J$:
\begin{enumerate}
\item There is $\p \in R$,  taking the role of a prime element, \st each left and each right ideal is of the form $J^n = \p^n R = R \p^n$.
\item $R$ is right hereditary \iff $R$ is left hereditary \iff $J$ is not nilpotent.
\item The injective envelope $I$ of $R/J$, as a right module, is in fact a faithfully balanced $R$-$R$-bimodule and the adjoint pair
\[ \Hom_R(-,I): \Mod R \rightleftharpoons R\lMod: \Hom_R(-,I) \]
restricts to a duality between the categories of left and right finite length modules.
\end{enumerate}
\end{Lem}

\begin{proof}
(1) follows directly from~\cite[Corollary B.1.4]{CF}, while (2) is a consequence of~\cite[Corollary B.1.5]{CF}. For (3), first note that~\cite[Proposition B.2.1 and Corollary B.2.2]{CF} imply that $I$ is semiartinian. Moreover, if $\hat R$ is the completion of $R$, then the restriction of constants $\Mod\hat R \to \Mod R$ clearly induces an equivalence between the categories of semiartinian $\hat R$-modules and semiartinian $R$-modules. Therefore, we can assume that $R$ is an Asano ring and the statement follows from~\cite[Proposition 2.3]{CF}.
\end{proof}

Having this, note that given any $G \in \Mod R$ we can define $\p^\sigma G$ and $G[r]$, $r \in R$, almost as in Section~\ref{sec:balanced}, only by right multiplication, and we get submodules of $G$. In fact everything in Sections~\ref{sec:balanced} and~\ref{sec:adic_topology} goes through and we get a generalization of Theorem~\ref{thm:K-pure}:

\begin{Prop} \label{prop:K-pure_general}
Let $R$ be an hereditary noetherian uniserial ring and $\lambda$ be an uncountable regular cardinal. Then there is a semiartinian right $R$-module $P_\lambda$ whose $\lambda$-pure projective dimension is at least two.
\end{Prop}

Now we can give a proof for Theorem~\ref{thm:grothendieck_catg}.

\begin{proof}[Proof of Theorem~\ref{thm:grothendieck_catg}]
The main idea is easy. Given a Grothendieck category $\A$ with the assumed properties, we wish to construct a finitely accessible subcategory $\B \subseteq \A$ and an hereditary Asano ring $R$ \st $\B$ is equivalent to the category of all semiartinian right $R$-modules. If we succeed, we learn that $\lpgldim \B \ge 2$ and by Lemma~\ref{lem:context_change} also
\[ \lpgldim\A \ge 2. \]

Assume first that $\A$ contains a tube and take $S \in \fp\A$ as in Definition~\ref{def:tube}. Let $\ES \subseteq \A$ be the full subcategory of objects admitting an ascending chain of subobjects
\[ 0 = X_0 \subseteq X_1 \subseteq \cdots \subseteq X_\ell = X \]
with $X_{i+1}/X_i \cong S$ for each $0 \le i < \ell$. Using an obvious generalization of Ringel's simplification lemma~\cite[3.1.1, pp. 114--115]{Ri}, one shows that $\ES$ is an abelian subcategory and $S$ is (up to isomorphism) the unique simple object of $\ES$. The term ``simplification'' comes from the latter fact---$S$ is made simple. By a result of Gabriel~\cite[\S8.3]{Gab2}, presented in detail in~\cite[Theorem 1.7.1]{ChKr}, $\ES$ is a uniserial category. That is, $\ES$ is an abelian length category whose each indecomposable object is uniserial. Now we can put $\B = \varinjlim\ES$, this is a locally finite Grothendieck category and it is by definition a finitely accessible subcategory of $\B$.

Next, we construct $R$. Namely, let $E$ be an injective envelope of $S$ in $\B$ and $R = \End_\B(E)$. As stated in~\cite{AmRi2}, $R$ is an hereditary Asano ring. In fact, this follows from the proof of \cite[Th\'eor\`eme 4, p. 398]{Gab}, as well as the fact that
\[ \Hom_\B(-,E)\colon \ES \la R\lMod \]
gives a duality between $\ES$ and the category of finite length left $R$-modules. Composing this with the duality from Lemma~\ref{lem:asano}(3), we get an equivalence between $\ES$ and the category of finite length right $R$-modules. It is a well known fact that one can extend this to an equivalence between $\B = \varinjlim\ES$ and the category of semiartinian right $R$-modules; see for instance again the Representation Theorem~\cite[2.26]{AR}. Thus, the first case is proved.

Let us now turn to the case when $\G$ is strictly wild. Then there is a field $k$ and a fully faithful functor
\[ \rmod\KronAlg \la \fp\A, \]
which again extends to a fully faithful and direct limit preserving functor 
\[ \Mod\KronAlg \la \A. \]
Moreover, the essential image of the latter functor is a finitely accessible subcategory of $\A$. The proof of Theorem~\ref{thm:examples} then tells us that the category of all torsion $\powS$-modules embeds as a finitely accessible subcategory in $\Mod\KronAlg$, hence also in $\A$.
\end{proof}

% ==============================================================================
\section{Consequences for triangulated categories}
\label{sec:conseq}

In the last section we focus on the original motivation for our results in the theory of triangulated categories and we discuss consequences for representable functors. A nice overview of the questions we are dealing with is presented in~\cite{Muro} and the general background is explained in the introduction of~\cite{ChKN}. In this context, our results give information about the derived categories of the rings and locally finitely presentable Grothendieck categories mentioned in the previous sections. 

To start with, recall that the derived category $\Der{R}$ of a ring $R$ is always a so-called compactly generated triangulated category:

\begin{Def} \label{def:comp_gen}
An object $X$ in a triangulated category $\T$ with arbitrary coproducts is called \emph{compact} if the representable functor
\[ \Hom_\T(X,-)\colon \T \la \Ab \]
sends coproducts in $\T$ to directed sums of abelian groups. The category $\T$ is said to be \emph{compactly generated} if

\begin{enumerate}
\item The full subcategory $\T^c$ of all compact objects is skeletally small;
\item Given $0 \ne X \in \T$, there is a non-zero homomorphism $C \to X$ in $\T$ with $C \in \T^c$.
\end{enumerate}
\end{Def}

Recall that $X \in \Der{R}$ is compact \iff $X$ is a \emph{perfect} complex, that is, $X$ is isomorphic in $\Der{R}$ to a bounded complex of finitely generated projective modules. One implication is very easy, while the other follows from the proof of~\cite[Proposition 6.3]{Ric}.
% Alternatively: \cite[Lemma~2.2]{Nee4} and \cite[Proposition~3.4]{BN}.

Compactly generated triangulated categories may be viewed as a triangulated analogue of abelian categories which are locally finitely presentable in the sense of~\cite{GaUl}. There is also a reasonable analogue of general locally presentable abelian categories, namely so called \emph{well generated} triangulated categories introduced by Neeman~\cite{Nee}. We will not define them here, we refer to~\cite[Chapter 8]{Nee} instead. We just mention that one replaces compact objects by so-called \emph{$\lambda$-compact} objects, where $\lambda$ is a regular cardinal; see~\cite[\S4.2]{Nee}. The full subcategory of $\T$ formed by all $\lambda$-compact objects is denoted by $\T^\lambda$, and by definition one has $\T^{\aleph_0} = \T^c$.

A special instance of~\cite[Lemma 4.4.5]{Nee} together with the (proof of) \cite[8.4.2]{Nee} gives us a description of $\lambda$-compact objects in~$\Der{R}$:

\begin{Lem} \label{lem:kappa-comp}  
Let $R$ be a ring and $\lambda$ an uncountable regular cardinal. Then the category $\Der{R}^\lambda$ of all $\lambda$-compact objects is the smallest full triangulated subcategory of $\Der{R}$ containing $R$ and closed under coproducts with fewer than $\lambda$ summands.
\end{Lem}

\begin{sloppypar}
Now an important question, for a compactly or well generated triangulated category, is when we can represent functors. That is, given a triangulated subcategory $\ES \subseteq \T$ and a covariant or contravariant functor
\[ F\colon \ES \la \Ab, \]
we ask for an object $X \in \T$ \st $F = \Hom_\T(X,-) |_{\ES}$ or $F = \Hom_\T(-,X) |_{\ES}$, respectively. Here and also later on, we always assume that all functors are additive. Of course, some extra assumptions are necessary since Hom-functors are also (co)homological and they transfer products or coproducts in $\T$ to products of abelian groups.
\end{sloppypar}

Simplest to state is the case when $\ES = \T$. This was studied first by Brown~\cite{Bro} in algebraic topology. We adopt the following notation:

\begin{enumerate}
\item[{[BR]}] $\T$ is said to satisfy \emph{Brown representability}, [BR] for short, if every contravariant cohomological functor $F\colon \T \to \Ab$ which sends coproducts to products is isomorphic to $\Hom_R(-,X)$ for some $X \in \T$.

\item[{[BR$^*$]}] $\T$ satisfies \emph{Brown representability for the dual} if every covariant homological functor $F'\colon \T \to \Ab$ which preserves products is isomorphic to $\Hom_R(X',-)$ for some $X' \in \T$.
\end{enumerate}

A less understood problem is the case when $\ES = \T^\lambda$, the category of $\lambda$-compact objects, for some regular cardinal $\lambda$. In topology, positive results for $\lambda=\aleph_0$ were obtained by Brown~\cite{Bro} and Adams~\cite{Adams}.
Note that the terminology below is not completely unified, we refer to~\cite[Remark 0.3]{ChKN} for an explanation.

\begin{enumerate}
\item[{[ARO$^\lambda$]}] $\T$ is said to satisfy \emph{Adams $\lambda$-representability for objects} if every contravariant cohomological functor $F\colon \T^\lambda \to \Ab$ which sends coproducts with fewer than $\lambda$ summands to products is isomorphic to $\Hom_\T(-,X) |_{\T^\lambda}$ for some $X \in \T$.

\item[{[ARM$^\lambda$]}] $\T$ satisfies \emph{Adams $\lambda$-representability for morphisms} if every natural transformation
\[ \eta\colon \Hom_\T(-,X) |_{\T^\lambda} \la \Hom_\T(-,Y) |_{\T^\lambda} \]
is induced by a (non-unique) morphism $X \to Y$ in $\T$.
\end{enumerate}

The questions about representability are intimately related to abelianizations of the triangulated category $\T$ in the sense of~\cite[Definition 6.1.3]{Nee} (see also~\cite[\S6.7]{Kr} and~\cite[\S10]{Bel}). The general motivation for approximating a triangulated category by an abelian category is that abelian categories are often better understood; see~\cite[Introduction]{Nee}. Let us give a precise definition here.

\begin{Def} \label{def:kappa_abelianization}
Let $\T$ be a triangulated category with arbitrary coproducts and $\lambda$ a regular cardinal. Denote by $\A_\lambda(\T)$ the category whose objects are contravariant functors
\[ F\colon \T^\lambda \la \Ab \]
which send coproducts with fewer than $\lambda$ summands to products, and whose morphisms are natural transformations between these functors. The \emph{$\lambda$-abelianization} of $\T$ is defined as the Yoneda functor
\[ H_\lambda\colon \T \la \A_\lambda(\T), \]
sending $X \in \T$ to $\Hom_\T(-,X) |_{\T^\lambda}$.
\end{Def}

Clearly, [ARM$^\lambda$] is equivalent to saying that $H_\lambda$ is full, and [ARO$^\lambda$] means just that every cohomological functor $\T^\lambda \to \Ab$ in $\A_\lambda(\T)$ belongs to the essential image of $H_\lambda$. What is better, $H_\lambda$ is a natural object determined by a universal property.

\begin{Prop} \label{prop:kappa_abel_univ} \cite{Nee,Kr}
Let $\T$ be a well generated triangulated category and $\lambda$ a regular cardinal \st $\T$ is generated by $\T^\lambda$. Then:

\begin{enumerate}
\item $\A_\lambda(\T)$ is a locally $\lambda$-presentable abelian category with enough projectives, exact products and coproducts and exact $\lambda$-filtered colimits. The functor $H_\lambda$ is homological, preserves products and coproducts and reflects isomorphisms.

\item Given any homological functor $H'\colon \T \to \A$ to an abelian category with coproducts and exact $\lambda$-filtered colimits \st $H'$ preserves coproducts, there exists an essentially unique coproduct preserving exact functor $E\colon \A_\lambda(\T) \to \A$ \st $H' = E \circ H_\lambda$.
\end{enumerate}
\end{Prop}

\begin{proof}
$\A_\lambda(\T)$ is locally $\lambda$-presentable by~\cite[\S\S6.9 and 6.10]{Kr} and has enough projectives by~\cite[Lemma 6.4.4]{Nee}. The fact that it has products and coproducts and these are exact has been proved in~\cite[\S6.3]{Nee}. The same for $\lambda$-filtered colimits follows easily from the observation that $\lambda$-filtered colimits in $\Ab$ commute with products with fewer than $\lambda$ summands. The functor $H_\lambda$ is clearly homological and preserves products. It preserves coproducts by~\cite[Proposition 6.2.6]{Nee} (see also~\cite[Proposition 6.7.1]{Kr}). It reflects isomorphisms by the assumption on $\T^\lambda$, see~\cite[Lemma 6.2.9]{Nee}. Finally, for the universal property (2) we refer to~\cite[Theorem B.2.5]{Nee} or~\cite[\S6.10]{Kr}.
\end{proof}

Let us now relate properties of the $\lambda$-abelianization of $\Der{R}$ and representability properties of $\Der{R}$ to $\lambda$-pure global dimension of $\Mod{R}$ and our Theorems~\ref{thm:examples} and~\ref{thm:grothendieck_catg}. We start with a result originally established by Neeman~\cite{Nee3} and Beligiannis~\cite[\S11.2]{Bel} for $\lambda = \aleph_0$ and extended by Muro and Ravent\'os~\cite{Rave} to higher cardinalities.

\begin{Prop} \label{prop:obstructions_repres} \cite{Nee3,Bel,Rave}
Let $\T$ be a well generated triangulated category and $\lambda$ a regular cardinal \st $\T$ is generated by $\T^\lambda$. Then:

\begin{enumerate}
\item If $F\colon \T^\lambda \to \Ab$ is a cohomological functor in $\A_\lambda(\T)$ \st $\pd_{\A_\lambda(\T)} F \le 2$, then $F \cong H_\lambda X$ for some $X \in \T$.
\item $H_\lambda\colon \T \to \A_\lambda(\T)$ is full, or equivalently \emph{[ARM$^\lambda$]} is satisfied for $\T$, \iff $\pd_{\A_\lambda(\T)} H_\lambda X \le 1$ for all $X \in \T$.
% In such a case also \emph{[ARO$^\lambda$]} is satisfied for $\T$.
\end{enumerate}
\end{Prop}

\begin{proof}
Part (2) is a direct generalization of~\cite[Lemma 4.1 and Proposition 4.11]{Nee3}, while for part (1) one can use the argument in~\cite[Remark 11.12]{Bel}. In both cases one also uses the fact that for each $X \in \Add\T^\lambda$ and $Y \in \T$ we have $\Hom_\T(X,Y) \cong \Hom_{A_\lambda(\T)}(H_\lambda X, H_\lambda Y)$. This easily follows from Definition~\ref{def:kappa_abelianization} for $X \in \T^\lambda$, and for a general $X \in \Add\T^\lambda$ this is true since $H_\lambda$ preserves coproducts.
\end{proof}

Next we need to give a connection between projective dimension of objects in $\A_\lambda\big(\Der{R}\big)$ and the $\lambda$-pure global dimension of $\Mod{R}$. Here we use a result of Muro and Ravent\'os~\cite{Rave}, which extends~\cite[Proposition 1.4]{ChKN} by Christensen, Keller and Neeman from $\aleph_0$ to arbitrary cardinals.

\begin{Prop} \label{prop:dim_related} \cite{ChKN,Rave}
Let $R$ be a right coherent ring \st each finitely presented right $R$-module is of finite projective dimension. Let further $\T = \Der{R}$ and $\lambda$ be a regular cardinal.

\begin{enumerate}
\item Given any $R$-module $M$, the projective dimension of $H_\lambda M$ in $\A_\lambda(\T)$ equals the $\lambda$-pure projective dimension of $M$ in $\Mod{R}$. In particular we have
\[ \qquad\quad
\lpgldim \Mod{R} \le \sup\{ \pd_{\A_\lambda(\T)} H_\lambda X \mid X \in \Der{R} \}.
\]

\item If $R$ is right hereditary, then
\[ \qquad\quad
\lpgldim \Mod{R} = \sup\{ \pd_{\A_\lambda(\T)} H_\lambda X \mid X \in \Der{R} \}.
\]
\end{enumerate}
\end{Prop}

\begin{proof}
Since~\cite{Rave} was not generally available at the time of writing of this text, we sketch an argument very similar to the one in~\cite{ChKN}. We consider the restriction
\[
H'_\lambda\colon \Mod{R} \la \A_\lambda(\T), \\
\]
of $H_\lambda$ to $\Mod{R}$, so that $H'_\lambda(M) = \Hom_\T(-,M) |_{\T^\lambda}$, and for each $i \in \Z$ the ``extended homology'' functors
\[
H^i\colon \A_\lambda(\T) \la \Mod{R},
\]
defined as $H^i(F) = F(R[-i])$. Note that $H^i \circ H_\lambda\colon \Der{R} \to \Mod{R}$ is the usual $i$-th homology functor and that $H^0 \circ H'_\lambda = 1_{\Mod{R}}$.

With this notation, we claim that $H'_\lambda$ as well as all $H^i$ commute with $\lambda$-directed (so also $\lambda$-filtered) colimits. In fact, a straightforward modification of the proof for \cite[Lemma 1.3]{ChKN} applies. For $H'_\lambda$, we must prove that given any $\lambda$-directed system $(M_i \mid i \in I)$ of $R$-modules and any $\lambda$-compact object $P \in \T$, the natural morphism
\[
\varinjlim \Hom_\T(P,M_i) \la \Hom_\T(P, \varinjlim M_i)
\]
is an isomorphism. This is clear for $P = R[i]$, $i \in \Z$, and for arbitrary $\lambda$-compact object $P$ we only use Lemma~\ref{lem:kappa-comp} and the fact that $\lambda$-directed colimits of abelian groups commute with products of fewer than $\lambda$ summands. The functors $H^i$ commute with $\lambda$-directed colimits since they are evaluation functors and $\lambda$-directed colimits are computed componentwise in $\A_\lambda(\T)$. This proves the claim.

Further, $H'_\lambda$ as well as all $H^i$ preserve coproducts and transform $\lambda$-pure exact sequences into $\lambda$-pure exact sequences. Namely, for coproducts of $\lambda$-presentable objects with fewer than $\lambda$ summands this follows from their construction (in $\A_\lambda(\T)$ they are \emph{not} computed componentwise, see~\cite[\S6.3]{Nee}!), and this extends to arbitrary coproducts via $\lambda$-directed colimits. For $\lambda$-pure exact sequences, we use the well-known fact that each $\lambda$-pure exact sequence is a $\lambda$-directed colimit of split exact sequences.

Finally, we check that $H'_\lambda$ sends $\lambda$-pure projective modules to projective objects and all $H^i$ send projective objects to $\lambda$-pure projective modules. This is an easy consequence of the fact that each $M \in \Mod{R}$ is represented by its projective resolution in $\T = \Der{R}$. Hence, if $M$ is a $<\lambda$-presented module, it is $\lambda$-compact in $\Der{R}$ using the assumptions on $R$ and Lemma~\ref{lem:kappa-comp}. On the other hand, if $X \in \T$ is $\lambda$-compact, the all its homologies are $<\lambda$-presented, again using the coherency of $R$ and Lemma~\ref{lem:kappa-comp}.

Having established all these properties, (1) and (2) are easy. They follow by exactly the same proof as for~\cite[Proposition 1.4]{ChKN}.
\end{proof}

Now we can give an application of our Theorem~\ref{thm:examples}. Partial results in this direction have been obtained by Muro and Ravent\'os, who studied wild hereditary algebras (for which our results also work, cf.\ Theorem~\ref{thm:grothendieck_catg}).

\begin{Cor} \label{cor:BA-repres}
Let $R$ be one of the rings from Theorem~\ref{thm:examples}, that is $R = \KronAlg$ or $R = k[x,y]$ for $k$ uncountable, or $R = \powSxy$ for any field $k$. Then there is \emph{no} regular cardinal $\lambda$ for which $\Der{R}$ satisfies \emph{[ARM$^\lambda$]}, the Adams $\lambda$-representability for morphisms. Rephrasing this, none of the $\lambda$-abelianization functors $H_\lambda\colon \Der{R} \to \A_\lambda\big(\Der{R}\big)$ is full.
\end{Cor}

\begin{Rem} \label{rem:BA-repres}
Based on~\cite{Ros} (see also~\cite[Conjecture 1.27]{Nee2}), it was believed that for nice enough (and perhaps even for all) well generated triangulated categories $\T$ there was a regular cardinal $\lambda$ \st $H_\lambda\colon \T \to \A_\lambda(\T)$ was full. ``Nice enough'' means that $\T$ is the homotopy category of a combinatorial stable model category in the sense of~\cite{HPS}; in this case $\T$ is automatically well generated by~\cite{Ros}. Here we present a counterexample to that belief, since $\Der{R}$ for any ring $R$ is certainly ``nice enough.''
\end{Rem}

\begin{Rem} \label{rem:BA-repres-objects}
Having ruled out [ARM$^\lambda$] for all regular cardinals for some $\Der{R}$, we say nothing about [ARO$^\lambda$], the Adams $\lambda$-representability for objects. Giving a counterexample to that for $\lambda = \aleph_0$ in~\cite{ChKN} was a rather delicate matter. Propositions~\ref{prop:obstructions_repres} and~\ref{prop:dim_related} suggest that one should look for $\lambda$ and an $R$-module $M$ \st $\lppd M \ge 3$, but we are at the present time not aware of any techniques for achieving this goal.
\end{Rem}

Finally, we mention a result by Neeman which brought considerable interest to studying fullness of $H_\lambda$. Having established [BR] for any well generated triangulated category $\T$ in~\cite{Nee}, he obtained in~\cite[Theorems 1.11 and 1.17]{Nee2} a result about [BR$^*$], the Brown representability for the dual. Here we in fact do not state the most general version of the result, but only a simplified version relevant for us.

\begin{Prop} \label{prop:rosicky_to_brown}
Let $\T$ be a well generated triangulated category and suppose there is a regular cardinal $\lambda$ such that $H_\lambda$ is full. Then $\T$ satisfies both \emph{[BR]} and~\emph{[BR$^*$]}.
\end{Prop}

As we now know, there need not exist such $\lambda$. However, this does not mean that [BR$^*$] is not satisfied for $\T$, existence of such $\lambda$ is only sufficient. Another sufficient condition is that $\A_\lambda(\T)$ has enough injectives for some $\lambda$, see~\cite[Theorem 8.6.1]{Nee}. The latter is satisfied for any compactly generated triangulated category $\T$ since then $\A_{\aleph_0}(\T)$ is a Grothendieck category. In particular, [BR$^*$] is satisfied for $\Der{R}$ for any ring $R$. On the other hand, examples of $\A_\lambda(\T)$ without enough injectives were given in~\cite[Appendix C.4]{Nee}, and the construction interestingly enough also uses Walker's modules $P_\beta$. To conclude with, the problem whether [BR$^*$] holds or not for any well generated triangulated category is to our best knowledge open.

% ==============================================================================
\bibliographystyle{abbrv}
\bibliography{counterex_rosicky_bib}

\begin{thebibliography}{10}

\bibitem{AR}
J.~Ad{\'a}mek and J.~Rosick{\'y}.
\newblock {\em Locally presentable and accessible categories}, volume 189 of
  {\em London Mathematical Society Lecture Note Series}.
\newblock Cambridge University Press, Cambridge, 1994.

\bibitem{Adams}
J.~F. Adams.
\newblock A variant of {E}. {H}. {B}rown's representability theorem.
\newblock {\em Topology}, 10:185--198, 1971.

\bibitem{AmRi2}
I.~K. Amdal and F.~Ringdal.
\newblock Cat\'egories unis\'erielles.
\newblock {\em C. R. Acad. Sci. Paris S\'er. A-B}, 267:A247--A249, 1968.

\bibitem{AF}
F.~W. Anderson and K.~R. Fuller.
\newblock {\em Rings and categories of modules}, volume~13 of {\em Graduate
  Texts in Mathematics}.
\newblock Springer-Verlag, New York, second edition, 1992.

\bibitem{ARS}
M.~Auslander, I.~Reiten, and S.~O.
\newblock {\em Representation theory of {A}rtin algebras}, volume~36 of {\em
  Cambridge Studies in Advanced Mathematics}.
\newblock Cambridge University Press, Cambridge, 1995.

\bibitem{BBL1}
D.~Baer, H.~Brune, and H.~Lenzing.
\newblock A homological approach to representations of algebras. {I}. {T}he
  wild case.
\newblock {\em J. Pure Appl. Algebra}, 24(3):227--233, 1982.

\bibitem{BBL2}
D.~Baer, H.~Brune, and H.~Lenzing.
\newblock A homological approach to representations of algebras. {II}. {T}ame
  hereditary algebras.
\newblock {\em J. Pure Appl. Algebra}, 26(2):141--153, 1982.

\bibitem{Bel}
A.~Beligiannis.
\newblock Relative homological algebra and purity in triangulated categories.
\newblock {\em J. Algebra}, 227(1):268--361, 2000.

\bibitem{Bou}
N.~Bourbaki.
\newblock {\em \'{E}l\'ements de math\'ematique. {T}opologie g\'en\'erale.
  {C}hapitres 1 \`a 4}.
\newblock Hermann, Paris, 1971.

\bibitem{BrGoe}
G.~Braun and R.~G\"{o}bel.
\newblock Splitting kernels into small summands.
\newblock To appear in Israel J. Math., 2010.

\bibitem{Bro}
E.~H. Brown, Jr.
\newblock Cohomology theories.
\newblock {\em Ann. of Math. (2)}, 75:467--484, 1962.

\bibitem{Bu}
T.~B{\"u}hler.
\newblock Exact categories.
\newblock {\em Expo. Math.}, 28(1):1--69, 2010.

\bibitem{ChKr}
X.-W. Chen and H.~Krause.
\newblock Introduction to coherent sheaves on weighted projective lines.
\newblock Preprint, arXiv:0911.4473v3, 2010.

\bibitem{ChKN}
J.~D. Christensen, B.~Keller, and A.~Neeman.
\newblock Failure of {B}rown representability in derived categories.
\newblock {\em Topology}, 40(6):1339--1361, 2001.

\bibitem{CF}
R.~R. Colby and K.~R. Fuller.
\newblock {\em Equivalence and duality for module categories}, volume 161 of
  {\em Cambridge Tracts in Mathematics}.
\newblock Cambridge University Press, Cambridge, 2004.
\newblock With tilting and cotilting for rings.

\bibitem{CH1}
P.~Crawley and A.~W. Hales.
\newblock The structure of torsion abelian groups given by presentations.
\newblock {\em Bull. Amer. Math. Soc.}, 74:954--956, 1968.

\bibitem{CH2}
P.~Crawley and A.~W. Hales.
\newblock The structure of abelian {$p$}-groups given by certain presentations.
\newblock {\em J. Algebra}, 12:10--23, 1969.

\bibitem{Fu}
L.~Fuchs.
\newblock {\em Infinite abelian groups. {V}ol. {II}}.
\newblock Academic Press, New York, 1973.
\newblock Pure and Applied Mathematics. Vol. 36-II.

\bibitem{Gab}
P.~Gabriel.
\newblock Des cat\'egories ab\'eliennes.
\newblock {\em Bull. Soc. Math. France}, 90:323--448, 1962.

\bibitem{Gab2}
P.~Gabriel.
\newblock Indecomposable representations. {II}.
\newblock In {\em Symposia {M}athematica, {V}ol. {XI} ({C}onvegno di {A}lgebra
  {C}ommutativa, {INDAM}, {R}ome, 1971)}, pages 81--104. Academic Press,
  London, 1973.

\bibitem{GaUl}
P.~Gabriel and F.~Ulmer.
\newblock {\em Lokal pr\"asentierbare {K}ategorien}.
\newblock Lecture Notes in Mathematics, Vol. 221. Springer-Verlag, Berlin,
  1971.

\bibitem{Hill2}
P.~Hill.
\newblock On the classification of abelian groups.
\newblock {\em Photocopied manuscript}, 1967.

\bibitem{Hill1}
P.~Hill.
\newblock A countability condition for primary groups presented by relations of
  length two.
\newblock {\em Bull. Amer. Math. Soc.}, 75:780--782, 1969.

\bibitem{HPS}
M.~Hovey, J.~H. Palmieri, and N.~P. Strickland.
\newblock Axiomatic stable homotopy theory.
\newblock {\em Mem. Amer. Math. Soc.}, 128(610):x+114, 1997.

\bibitem{JL}
C.~U. Jensen and H.~Lenzing.
\newblock {\em Model-theoretic algebra with particular emphasis on fields,
  rings, modules}, volume~2 of {\em Algebra, Logic and Applications}.
\newblock Gordon and Breach Science Publishers, New York, 1989.

\bibitem{Kap}
I.~Kaplansky.
\newblock The homological dimension of a quotient field.
\newblock {\em Nagoya Math. J.}, 27:139--142, 1966.

\bibitem{Kel}
B.~Keller.
\newblock Chain complexes and stable categories.
\newblock {\em Manuscripta Math.}, 67(4):379--417, 1990.

\bibitem{Kern}
O.~Kerner.
\newblock Representations of wild quivers.
\newblock In {\em Representation theory of algebras and related topics
  ({M}exico {C}ity, 1994)}, volume~19 of {\em CMS Conf. Proc.}, pages 65--107.
  Amer. Math. Soc., Providence, RI, 1996.

\bibitem{KL}
L.~Klingler and L.~S. Levy.
\newblock Representation type of commutative {N}oetherian rings (introduction).
\newblock In {\em Algebras, rings and their representations}, pages 113--151.
  World Sci. Publ., Hackensack, NJ, 2006.

\bibitem{Kr}
H.~Krause.
\newblock Localization theory for triangulated categories.
\newblock In {\em Triangulated categories}, volume 375 of {\em London Math.
  Soc. Lecture Note Ser.}, pages 161--235. Cambridge Univ. Press, Cambridge,
  2010.

\bibitem{Kul}
L.~Y. Kulikov.
\newblock Generalized primary groups. {I}.
\newblock {\em Trudy Moskov. Mat. Ob\v s\v c.}, 1:247--326, 1952.

\bibitem{Lenz2}
H.~Lenzing.
\newblock Homological transfer from finitely presented to infinite modules.
\newblock In {\em Abelian group theory ({H}onolulu, {H}awaii, 1983)}, volume
  1006 of {\em Lecture Notes in Math.}, pages 734--761. Springer, Berlin, 1983.

\bibitem{Lenz}
H.~Lenzing.
\newblock The pure-projective dimension of torsion-free divisible modules.
\newblock {\em Comm. Algebra}, 12(5-6):649--662, 1984.

\bibitem{Mi}
R.~Mines.
\newblock A family of functors defined on generalized primary groups.
\newblock {\em Pacific J. Math.}, 26:349--360, 1968.

\bibitem{Muro}
F.~Muro.
\newblock Representability of cohomology theories.
\newblock Talk at CSASC 2010 in Prague, presentation available at
  http://personal.us.es/fmuro/praha.pdf.

\bibitem{Nee3}
A.~Neeman.
\newblock On a theorem of {B}rown and {A}dams.
\newblock {\em Topology}, 36(3):619--645, 1997.

\bibitem{Nee}
A.~Neeman.
\newblock {\em Triangulated categories}, volume 148 of {\em Annals of
  Mathematics Studies}.
\newblock Princeton University Press, Princeton, NJ, 2001.

\bibitem{Nee2}
A.~Neeman.
\newblock Brown representability follows from {R}osick\'y's theorem.
\newblock {\em J. Topol.}, 2(2):262--276, 2009.

\bibitem{Nun}
R.~J. Nunke.
\newblock Purity and subfunctors of the identity.
\newblock In {\em Topics in {A}belian {G}roups ({P}roc. {S}ympos., {N}ew
  {M}exico {S}tate {U}niv., 1962)}, pages 121--171. Scott, Foresman and Co.,
  Chicago, Ill., 1963.

\bibitem{Nun2}
R.~J. Nunke.
\newblock Homology and direct sums of countable abelian groups.
\newblock {\em Math. Z.}, 101:182--212, 1967.

\bibitem{Os}
B.~L. Osofsky.
\newblock {\em Homological dimensions of modules}.
\newblock American Mathematical Society, Providence, R. I., 1973.
\newblock Conference Board of the Mathematical Sciences Regional Conference
  Series in Mathematics, No. 12.

\bibitem{Rave}
O.~Ravent\'{o}s~Morera.
\newblock {\em {A}dams Representability in Triangulated Categories}.
\newblock PhD thesis, University of Barcelona, 2011.

\bibitem{Ric}
J.~Rickard.
\newblock Morita theory for derived categories.
\newblock {\em J. London Math. Soc. (2)}, 39(3):436--456, 1989.

\bibitem{Ri}
C.~M. Ringel.
\newblock {\em Tame algebras and integral quadratic forms}, volume 1099 of {\em
  Lecture Notes in Mathematics}.
\newblock Springer-Verlag, Berlin, 1984.

\bibitem{RosT}
J.~Rosick{\'y}.
\newblock Generalized purity, definability and {B}rown representability.
\newblock Talk at the conference STA'09 in Prague, presentation available at
  http://www.math.muni.cz/$\sim$rosicky/papers/praha.pdf.

\bibitem{Ros}
J.~Rosick{\'y}.
\newblock Generalized {B}rown representability in homotopy categories.
\newblock {\em Theory Appl. Categ.}, 14:no. 19, 451--479 (electronic), 2005.
\newblock Revised version available at arXiv:math/0506168v3.

\bibitem{RosE}
J.~Rosick{\'y}.
\newblock Erratum: ``{G}eneralized {B}rown representability in homotopy
  categories'' [{T}heory {A}ppl. {C}ateg. {\bf 14} (2005), no. 19, 451--479].
\newblock {\em Theory Appl. Categ.}, 20:No. 2, 18--24, 2008.

\bibitem{S}
L.~Salce.
\newblock {\em Struttura dei {$p$}-gruppi abeliani}, volume~18 of {\em Quad.
  dell'Unione Matematica Italiana}.
\newblock Pitagora, Bologna, 1980.

\bibitem{Tr}
J.~Trlifaj.
\newblock {B}rown representability test problems in locally {G}rothendieck
  categories.
\newblock To appear in Appl. Categ. Structures, 2010.

\bibitem{VRos}
A.-C. Van~Roosmalen.
\newblock Abelian hereditary fractionally {C}alabi-{Y}au categories.
\newblock Preprint, arXiv:1008.1245v2, 2010.

\bibitem{W}
E.~A. Walker.
\newblock The groups {$P_{\beta }$}.
\newblock In {\em Symposia {M}athematica, {V}ol. {XIII} ({C}onvegno di {G}ruppi
  {A}beliani, {INDAM}, {R}ome, 1974)}, pages 245--255. Academic Press, London,
  1974.

\end{thebibliography}

\end{document}